\documentclass[11pt, oneside]{article}   	
\usepackage{geometry}                		
\usepackage{graphicx}				
								
\usepackage{amsmath,amsthm,bm}
\usepackage{amssymb}
\usepackage{mathrsfs}
\usepackage{bbm}
\usepackage{caption}
\usepackage{subcaption}
\usepackage{mathtools}
\usepackage{hhline}

\usepackage{authblk}

\newtheorem{theorem}{Theorem}[section]

\newtheorem{proposition}{Proposition}[section]
\newtheorem{lemma}[theorem]{Lemma}

\newtheorem{remark}{Remark}[section]
\newcommand{\vertiii}[1]{{\left\vert\kern-0.25ex\left\vert\kern-0.25ex\left\vert #1 
    \right\vert\kern-0.25ex\right\vert\kern-0.25ex\right\vert}}

\title{Higher-order Accurate Spectral Density Estimation of Functional Time Series}

\author{Tingyi Zhu}
\author{Dimitris N. Politis}

\affil{Department of Mathematics\\
University of California San Diego\\
La Jolla, CA 92093-0112, USA}
\date{}                     

\begin{document}
\maketitle

\begin{abstract}

Under the frequency domain framework for weakly dependent functional time series, a key element is the spectral density kernel which encapsulates the second-order dynamics of the process. We propose a class of spectral density kernel estimators based on the notion of a flat-top kernel. The new class of estimators employs the inverse Fourier transform of a flat-top function as the weight function employed to smooth the periodogram. It is shown that using a flat-top kernel yields a bias reduction and results in a higher-order accuracy in terms of optimizing the integrated mean square error (IMSE). Notably, the higher-order accuracy of flat-top estimation comes at the sacrifice of the positive semi-definite property. Nevertheless, we show how a flat-top estimator can be modified to become positive semi-definite (even strictly positive definite) in finite samples while retaining its favorable asymptotic properties. In addition, we introduce a data-driven bandwidth selection procedure realized by an automatic inspection of the estimated correlation structure. Our asymptotic results are complemented by a finite-sample simulation where the higher-order accuracy of flat-top estimators is manifested in practice.
\\

\noindent
{\bf Keywords:} Functional time series; spectral density kernel; spectral density estimation flat-top kernel, positive semi-definite estimation
\end{abstract}

\section{Introduction}
Functional time series has become a recent focus within the statistical research of functional data analysis due to the fact that functional data are often collected sequentially over time. Typically, we consider a stationary functional sequence $\{X_t(\tau); \tau\in[0,1]\}_{t\in \mathbb{Z}}$ whose terms are random elements of the separable Hilbert space $L^2([0,1],\mathbb{R})$. The central issue in the analysis of functional time series  is to take into account the temporal dependence between the observations. That amounts to the investigation of  second-order characteristics of the functional sequences. 
A handful of  early papers have studied the covariance structure  of the functional sequences with dependence. For example,  
Bosq (2002) \cite{Bosq2002}, Dehling and Sharipov (2005) \cite{Deh2005} consider the estimation of covariance operator for functional autoregressive processes, Horv\'ath and Kokoszka (2010) \cite{Hor2010} studied  the  covariance structure of  weakly dependent functional time series under an $m$-dependence condition;
see also Horv\'ath and Kokoszka (2012) \cite{Hor2012} for an overview. 

Nevertheless, to obtain a complete description of the second-order structure of dependent functional sequences, one needs to consider autocovariance operators, or autocovariance kernels relating different lags of the series, analogous to the autocovariance matrices in the context of multivariate time series analysis.
One statistic of interest associated with the autocovariance operator is the long-run covariance kernel (or long-run covariance function) defined as $C(\tau,\sigma)=\sum_{l \in \mathbb{Z}} r_l(\tau,\sigma)$ where $r_t(\tau, \sigma)=\text{cov}(X_{t+s}(\tau),X_{s}(\sigma))$ for $\tau,\sigma\in [0,1]$  and $t,s\in \mathbb{Z}$, is the so-called autocovariance kernel.
The analysis of the long-run covariance kernel is applicable to general functional dependence sequences without a particular model assumption.

Horv\'ath et. al (2013) \cite{Hor2013} proposed the kernel lag-window estimator of $C(\tau,\sigma)$ and showed its consistency under mild conditions. 
The asymptotic normality of the estimator is established in Berkes et. al. (2016) \cite{Ber2016}.
The estimation has applications in mean and stationarity testing of functional time series, see Horv\'ath et. al (2015) \cite{Hor2015} and Jirak (2013) \cite{Jir2013}. 
Horv\'ath et. al (2016) \cite{Hor2016} and Rice and Shang (2016) \cite{Ric2016} address the bandwidth selection for the kernel of the lag-window estimator.

Rather than focus on the isolated characteristic like the long-run covariance, Panaretos and Tavakoli (2013) \cite{Pan2013} approach the problem of inferring the second-order structure of stationary functional time series via Fourier analysis, formulating a frequency domain framework for weakly dependent functional data. In the frequency domain of functional setting, the entire second-order dynamical properties are encoded in the spectral density kernel  which is defined as
\begin{equation}
f_\omega(\tau, \sigma)=\frac{1}{2\pi}\sum_{t\in \mathbb{Z}}\text{exp}(-i\omega t)r_t(\tau,\sigma),
\end{equation}

\noindent
where the autocovariance kernel $r_t(\tau, \sigma)$ and the spectral density kernel $f_\omega(\tau,\sigma)$ comprise a Fourier pair. The notion of spectral density kernel in the functional setting is a generalization of finite-dimensional notion in the context of spectral density analysis of multivariate time series, which has been extensively studied by prominent statistical researchers; see, e.g. Parzen (1957, 1961) \cite{Par1957,Par1961}, Brillinger and Rosenblatt (1967) \cite{Bri1967}, Hannan (1970) \cite{Han1970} and Priestley (1981) \cite{Pri1981}.
A consistent estimate of the spectral density kernel in the form of a weighted average of the periodogram kernel---the functional analogous of periodogram matrix---is also proposed  in Panaretos and Tavakoli (2013) \cite{Pan2013} under a type of cumulant mixing condition. This weak dependence condition is the functional analog of classical cummulant-type mixing condition of Brillinger (2001) \cite{Bri2001}.

In this paper, we propose a new class of spectral density kernel estimators based on the notion of flat-top kernel defined in Politis (2001) \cite{Pol2001}; see also Politis and Romano (1995, 1996, 1999) \cite{Pol1995,Pol1996,Pol1999}. The new class of estimators employs the inverse Fourier transform of a flat-top function to  construct the weight function smoothing the periodogram. With the choice of a  high-order flat-top kernel, it is shown to be able to achieve bias reduction, and hence the higher-order accuracy in terms of optimizing the integrated mean square error (IMSE). It is also nearly equal to the general lag-window type estimators which is a well-know fact in finite-dimensional case; see Brockwell and Richard (2013) \cite{Bro2013} and Brillinger (2001) \cite{Bri2001}. 

The higher-order accuracy of flat-top estimation typically comes at the sacrifice of the positive semi-definite property. To address this issue, we show how a flat-top estimator can be modified to become positive semi-definite (even strictly positive definite) while retaining the favorable asymptotic properties. The modification is similar to the one proposed in Politis (2011) \cite{Pol2011}, for the treatment of flat-top spectral density matrix estimators. In addition, we introduce a data driven bandwidth selection procedure realized by an automatic inspection of the correlation structure.

The structure of the paper is as follows. In the next section, the flat-top estimator of the spectral density kernel is defined after introduction of some basic definitions of the frequency domain framework, and theorems on the asymptotic accuracy are given. Section 3 shows the almost equivalence of the proposed estimator in the form of weighted average of periodogram and the flat-top lag-window estimator. A modification of the flat-top spectral density estimator is introduced in Section 4 which results into an estimator that is positive semi-definite while retaining the estimator's higher-order accuracy. Section 5 addresses the issue of data-dependent bandwidth choice where an empirical rule of picking bandwidth is proposed. Our favorable asymptotic results are supported by a finite-sample simulation in Section 7 where higher-order accuracy of the flat-top estimators are manifested in practice. Finally, the technical proofs are gathered in the Appendix in Section 7.

\section{Spectral density kernel estimation}

We consider a functional time series $\{X_t(\tau); \tau\in[0,1]\}_{t\in \mathbb{Z}}$ where each $X_t(\cdot)$ belongs to the separable Hilbert space $L^2([0,1],\mathbb{R})$ possessing mean zero, i.e., $\mathbb{E}X_t(\tau)=0$ for all $\tau\in[0,1]$, and autocovariance kernel $r_t(\tau, \sigma)=\mathbb{E}X_{t+s}(\tau)X_{s}(\sigma)$ for $\tau,\sigma\in [0,1]$ and $s\in \mathbb{Z}$.

The space is equipped with the inner product $\langle\cdot,\cdot\rangle$ and the induced $L^2$ norm $||\cdot||_2$,
$$\langle f,g\rangle=\int_0^1 f(\tau)g(\tau)d\tau,\qquad ||g||_2=\langle g,g\rangle^{1/2}, \qquad f,g\in L^2([0,1],\mathbb{R}). $$
We assume the series $\{X_t\}_{t\in \mathbb{Z}}$ is strictly stationary in the sense that for any finite set of indices $I \subset \mathbb{Z}$ and any $s \subset \mathbb{Z}$, the joint law of $\{X_t, t\in I\}$ is identical to that of $\{X_{t+s}, t\in I\}$.

In addition, the weak dependence structure among the observations $\{X_t\}$ is quantified by employing the notion of \emph{cumulant kernel} of a series. The pointwise definition of a $k$th \emph{order cumulant kernel} is 
$$\text{cum}(X_{t_1}(\tau_1),\dots, X_{t_k}(\tau_k))= \sum_{\nu=(\nu_1,\dots, \nu_p)}(-1)^{p-1}(p-1)!\prod_{l=1}^p\mathbb{E}\left[\prod_{j\in\nu_l}X_{t_j}(\tau_j)\right]$$
 where the sum extends over all unordered partitions of $\{1,\dots,k\}$. We will make use of  the following  cumulant mixing condition, defined for fixed $l\geq 0$ and $k=2,3,\dots$\\

\noindent
{\bf Condition $\mathbf{C(l,k)}$.}  For each $j=1,\dots,k-1,$
$$\sum_{t_1,\dots,t_{k-1}=-\infty}^{\infty} (1+|t_j|^l)||\text{cum}(X_{t_1},\dots, X_{t_{k-1}},X_0)||_2<\infty$$

We inherit the frequency domain framework of functional time series developed in Panaretos and Tavakoli (2013) \cite{Pan2013}, in which the functional version of discrete Fourier transform is introduced. Given a functional sequence of length $T$, $\{X_t\}_{t=0}^{T-1}$, the \emph{functional Discrete Fourier Transform} (fDFT) is defined as
\begin{equation}\label{eqn.peri1}                                                         
\tilde{X}_\omega^{(T)}(\tau)= (2\pi T)^{1/2}\sum_{t=0}^{T-1}X_t(\tau)\exp(-i\omega t).
\end{equation}
The tensor products of the fDFT leads to the notion of \emph{periodogram kernel}--the functional analogue of the periodogram matrix in the multivariate case. The \emph{periodogram kernel} is defined as 
\begin{equation}\label{eqn.peri2}    
p_\omega^{(T)}(\tau,\sigma)=\tilde{X}_\omega^{(T)}(\tau)\tilde{X}_{-\omega}^{(T)}(\sigma).
\end{equation}
The periodogram kernel is asymptotically unbiased under certain cumulant mixing conditions. However, it is not a consistent estimator of the spectral density kernel $f_\omega$ as its asymptotic covariance is not zero.  Panaretos and Tavakoli (2013) \cite{Pan2013} proposed a consistent estimator $f_{ \omega}^{(T)}$  by convolving the periodogram kernel with a weight function, which has the form 

\begin{equation}\label{f.T}
 f_{ \omega}^{(T)}(\tau,\sigma)=\frac{2\pi}{T}\sum_{s=1}^{T-1}W^{(T)}\left(\omega-\frac{2\pi s}{T}\right)p_{2\pi s/T}^{(T)}(\tau,\sigma).
\end{equation}

\noindent
The weight function, $W^{(T)}$, is constructed as 

\begin{equation}\label{W.T}
W^{(T)}(x)=\sum_{j\in \mathbb{Z}}\frac{1}{B_T}W\left(\frac{x+2\pi j}{B_T}\right)
\end{equation}

\noindent
where $B_T, T=1,2,\dots$ is a sequence of scale parameters with the properties $B_T>0, B_T\to 0$ and $B_TT\to \infty$ as $T\to \infty$. $W$ is a fixed function satisfying that $W$ is a positive, even function and
$$W(x)=0 \text{ for } |x|\geq 1; \;\;\; \int_{-\infty}^{\infty}W(x)dx=1;\;\;\; \int_{-\infty}^{\infty}W(x)^2dx<\infty.$$

\noindent
The summation over $j\in \mathbb{Z}$ in (\ref{W.T}) makes the weight function $W^{(T)}$ periodic with period $2\pi$. The same will be true for the estimator $f_\omega^{(T)}$ by its definition in (\ref{f.T}). With the above constraints imposed on function $W$, it has been shown in Panaretos and Tavakoli (2013) \cite{Pan2013} that $f_\omega^{(T)}$ is a consistent estimator of the spectral density kernel $f_\omega$ in mean square (with respect to Hilbert-Schmidt norm). The bias of $f_\omega^{(T)}$ is partly attributed to the assumption that  $W$ is positive and it  can potentially be significantly reduced if an appropriate function $W$ is chosen that is not restricted to be positive. To this aim, we propose a class of higher-order accurate estimator by making use of  the so-called flat-top kernels in the construction of weight function $W^{(T)}$. The resulting estimator is shown to achieve bias reduction while retaining the asymptotic covariance structure of $f_{ \omega}^{(T)}$ in (\ref{f.T}). 

To describe our estimator, we need the notion of a ``flat-top" kernel. A general flat-top kernel $\Lambda$ is defined in terms of its Fourier transform $\lambda$, which is in turn defined as
 
\begin{equation} 
  \lambda(s)=
   \begin{dcases*} 
      \;\;1 & if $|s|\leq c$,\\
      \;\;g(s) & otherwise.
   \end{dcases*}
\end{equation} 

\noindent
where $c>0$ is a parameter and $g: \mathbb{R}\to [-1,1]$ is a symmetric function, continuous at all but a finite number of points satisfying $g(c)=1$ and $\int_{\mathbb{R}}g^2(x)dx<\infty$. The flat-top kernel $\Lambda(x)$ is then given by the inverse Fourier transform of $\lambda(s)$
\begin{equation}\label{inv.fou.tran}
\Lambda(x) =\frac{1}{2\pi}\int^\infty_{-\infty}\lambda(s)e^{isx}ds.
\end{equation}
Note that in the preceding definition, the function $\lambda$, and hence the kernel $\Lambda$, depend on the function $g$ and the parameter $c$, but this dependence will not be explicitly denoted.

The function $\lambda(s)$ is `flat', i.e., constant, over the region $[-c,c]$, hence the name flat-top for the kernel function $\Lambda(x)$. 
If a kernel function $\Lambda$ has finite $q$th moment, and its moments up to order $q-1$ are equal to zero, i.e. $\int s^q\Lambda(s)ds<\infty$, and $\int s^k\Lambda(s)ds=0$ for all $k\leq q-1$, then the kernel is said to be of order $q$.
We have the following property concerning the order of the kernel function $\Lambda$:

\begin{proposition}\label{prop.order}
If $\lambda(s)$ is p times differentiable flat-top function and $\lambda^{(p)}$ is H\"older continuous of order $0<\alpha<1$, then $\Lambda(x)$ is a kernel of order $p-1$. 
\end{proposition}

\begin{proof}
See Appendix.
\end{proof}

In the following, we will be using $ \hat f_{\omega,\lambda}^{(T)}$ to denote the flat-top estimator employing the flat-top kernel $\Lambda$, which is in turn induced by a flat-top function $\lambda$. The estimator is of the same form as (\ref{f.T}), except the difference in the weight function, i.e.,

\begin{equation} \label{est}
 \hat f_{\omega,\lambda}^{(T)}(\tau,\sigma)=\frac{2\pi}{T}\sum_{s=1}^{T-1}W_\lambda^{(T)}\left(\omega-\frac{2\pi s}{T}\right)p_{2\pi s/T}^{(T)}(\tau,\sigma),
\end{equation}

\noindent
where 

\begin{equation} \label{weight}
W_\lambda^{(T)}(x)=\sum_{j\in \mathbb{Z}}\frac{1}{B_T}\Lambda\left(\frac{x+2\pi j}{B_T}\right),
\end{equation}

\noindent
with $\Lambda(x)$ being the flat-top kernel induced by a flat-top function $\lambda(s)$ defined in (\ref{inv.fou.tran}).\\

\noindent
The following theorems investigate the performance of $\hat f_{\omega,\lambda}^{(T)}$ employing the general flat-top kernel $\Lambda$.\\

\begin{theorem}\label{thm.bias}
Provided that $B_T\to 0$ and $B_TT\to \infty$ as $T\to \infty$, and assume $p$ is the maximum value that can be attained such that C(p,2) holds; then by choosing an appropriate flat-top kernel $\Lambda$ of order $p$, 
we have 
\begin{equation*}
\mathbb{E}[\hat f_{\omega,\lambda}^{(T)}(\tau, \sigma)]= f_\omega(\tau,\sigma)+ O(B_T^p)+O(B_T^{-1}T^{-1}),
\end{equation*}
where the equality holds in $L^2$, and the error terms are uniform in $\omega$.
\end{theorem}

\begin{proof}
See Appendix.
\end{proof}

\begin{remark} \rm According to Proposition \ref{prop.order}, a sufficient condition for a kernel $\Lambda$ to be of order $p$ is that $\lambda$ is a $p+1$ times differentiable flat-top function and $\lambda^{(p+1)}$ is H\"older continuous. On the other hand, the decay rate of bias crucially depends on the cumulant mixing condition satisfied by the functional sequence. 
A type of moment condition is provided in Panaretos and Tavakoli (2013) \cite{Pan2013},  which is sufficient for the cumulant mixing condition to hold for a general linear process of the form $X_t=\sum_{s\in\mathbb{Z}}A_s\varepsilon_{t-s}$; see Proposition 4.1 therein. 
\end{remark}

\begin{remark} \rm It is worth mentioning that Theorem \ref{thm.bias}-\ref{thm.mse} can hold for a non-flat-top kernel $\Lambda$ of order $p$ as long as it satisfies properties (i)-(iv) in Lemma 7.2. Nevertheless, in the paper we will be focusing on the flat-top kernels for the simplicity of the proof. In addition, an empirical data-driven bandwidth selection rule is proposed in Section 6, which can be applied exclusively on flat-top kernels.
\end{remark}

The flat-top estimator  $\hat f_{\omega,\lambda}^{(T)}$ achieves bias improvements while retaining the rate of decay of the covariance structure as stated in the following theorem:

\begin{theorem} \label{thm.cov}
Under C(1,2) and C(1,4),
\begin{equation*}
cov(\hat f_{\omega_1, \lambda}^{{(T)}}(\tau_1,\sigma_1),\hat f_{\omega_2, \lambda}^{(T)}(\tau_2,\sigma_2))=O(B_T^{-2}T^{-1})
\end{equation*}
where the equality holds in $L^2$, uniformly in the $\omega$'s.

\end{theorem}

\noindent
Using our Lemmas \ref{property.WT} and \ref{lemma.WT.approx}, Theorem \ref{thm.cov} can be proved along the same lines of proof of Corollary 3.3 in Panaretos and Tavakoli (2013) \cite{Pan2013}. For fixed $\omega_1,\omega_2$, the covariance can be shown to have a sharper bound $O(B_T^{-1}T^{-1})$; see Proposition 3.4 in Panaretos and Tavakoli (2013) \cite{Pan2013}. \\

Concerning the mean square error, we need the notion of \emph{spectral density operator} of functional time series, which is introduced  in Panaretos and Tavakoli (2013)\cite{Pan2013}. The spectral density operator $\mathscr{F}_\omega$ is an operator induced by the spectral density kernel through right integration,

\begin{equation}
\mathscr{F}_\omega h(\tau)= \int_0^1 f_\omega(\tau,\sigma)h(\sigma)d\sigma = \frac{1}{2\pi}\sum_{t\in \mathbb{Z}}e^{-i\omega t} \int_0^1  r_t(\tau,\sigma)h(\sigma)d\sigma
=\frac{1}{2\pi}\sum_{t\in \mathbb{Z}}e^{-i\omega t}\mathscr{R}_t h(\tau),
\end{equation} 
where $\mathscr{R}_t$ is the \emph{autocovariance operator} induced by the autocovariance kernel through right integration,
\begin{equation}
\mathscr{R}_t h(\tau)=\int_0^1 r_t(\tau,\sigma)h(\sigma)d\sigma=\text{cov}[\langle X_0,h\rangle,X_t(\tau)], \qquad \quad h\in L^2([0,1],\mathbb{R}).
\end{equation} 

The spectral density operator $\mathscr{F}_\omega$ is the integral operator with kernel $f_\omega$. Analogously, we denote $\hat{\mathscr{F}}_{\omega, \lambda}^{(T)}$ the operator induced by the the kernel $\hat f_{\omega,\lambda}^{(T)}$ through right integration, and thereby the estimator of $\mathscr{F}_\omega$. Combining the results on the asymptotic bias and variance of the spectral density operator, we have the following consistency in {\em integrated mean square}  of the induced estimator $\hat{\mathscr{F}}_{\omega, \lambda}^{(T)}$ for the spectral density operator $\mathscr{F}_\omega$.

\begin{theorem}\label{thm.mse} 
Provided assumptions C(p,2) and C(1,4) hold, $B_T\to 0$, $B_TT\to\infty$, then the spectral density operator estimator $\hat{\mathscr{F}}_{\omega,\lambda}^{(T)}$ employing a flat-top kernel $\Lambda$ of order $p$ is consistent in integrated mean square, that is, 
$$\text{IMSE}(\hat{\mathscr{F}}_{\omega,\lambda}^{(T)}) = \int_{-\pi}^{\pi}\mathbb{E}\vertiii{\hat{\mathscr{F}}_{\omega,\lambda}^{(T)}-\mathscr{F}_\omega}_2^2d\omega \to 0, \qquad \text{as} \;\; T\to \infty,$$
where $\vertiii{\cdot}$ is the Hilbert-Schmidt norm. More precisely, $\text{IMSE}(\hat{\mathscr{F}}_{\omega,\lambda}^{(T)})=O(B_T^{2p})+O(B_T^{-1}T^{-1})$ as $T\to\infty$.
\end{theorem}

Theorem \ref{thm.mse} gives the rate of convergence of $\hat{\mathscr{F}}_{\omega,\lambda}^{(T)}$ to $\mathscr{F}_\omega$. In the meantime, it also suggests the optimal value of the bandwidth parameter $B_T$ in terms of optimizing the decay rate of integrated mean square error. Apparently, the optimal $B_T$ depends on the cumulant condition a functional sequence possesses, that is, the value of $p$. For any finite $p$,  the optimal decay rate $O(T^{-2p/(2p+1)})$ can be achieved with $B_T=T^{-1/(2p+1)}$. In the case that $p=\infty$, one can choose $B_T=1/\log{T}$ to obtain a favorable rate of $O(\log T/T)$.

\section{Alternate estimates and flat-top kernel choice}
\subsection{Alternate estimates} \label{sec.alt.est}
The spectral density kernel estimator considered in the previous section has the form of a weighted average of periodogram ordinates. In fact, the weight function $W_\lambda^{(T)}(x)$ in the estimate (\ref{est}) has an alternate form

\begin{equation}\label{alter.W}
W_\lambda^{(T)}(x) = \frac{1}{2\pi}\sum_{u\in \mathbb{Z}}\lambda(B_Tu)e^{-ixu}. 
\end{equation}
The equivalence of (\ref{weight}) and (\ref{alter.W}) can be easily verified by using Poisson summation formula. Moreover, if the discrete average in (\ref{est}) is replaced by a continuous one, the estimate becomes
\begin{align} 
\int_0^{2\pi}W_{\lambda}^{(T)}(\omega-\alpha)p_{\alpha}^{(T)}(\tau,\sigma)d\alpha&=\int_0^{2\pi}W_{\lambda}^{(T)}(\alpha)p_{\omega-\alpha}^{(T)}(\tau,\sigma)d\alpha \nonumber\\
&=\int_{-\infty}^{\infty}B_T^{-1}\Lambda(B_T^{-1}\alpha)p_{\omega-\alpha}^{(T)}(\tau,\sigma)d\alpha. \label{eqn.alter}
\end{align}

\noindent
By Equation (\ref{eqn.peri1}) and (\ref{eqn.peri2}), the periodogram kernel is given by
\begin{align}
p_{\omega}^{(T)}(\tau,\sigma) &= \frac{1}{2\pi}\sum_{|u|<T}\hat r_u(\tau,\sigma) e^{-i\omega u}
\end{align}
where

\begin{equation*}
\hat r_u(\tau,\sigma)= \frac{1}{T}\sum_{0\leq t,t+u\leq T-1}X_{t+u}(\tau)X_{t}(\sigma)
\end{equation*}
is the sample autocovariance kernel. If this is substituted into (\ref{eqn.alter}), then the estimate takes the form 

\begin{equation}\label{gene.est}
\frac{1}{2\pi}\sum_{|u|<T}\lambda(B_T u)\hat r_u(\tau,\sigma)e^{-i\omega u}
\end{equation}
where 

\begin{equation*}
\lambda(s) =\int_{-\infty}^\infty\Lambda(x)e^{-isx}dx.
\end{equation*}

With a flat-top function $\lambda$ in place,
the estimate (\ref{gene.est}) has a formal resemblance to the flat-top estimation for multivariate spectrum explored in Politis (2011) \cite{Pol2011} with the bandwidth parameter $m_T=B_T^{-1}$. The estimator has been shown to achieve higher-order accuracy in estimating the spectral density matrix; see Politis (2011) \cite{Pol2011} for details. In fact, the estimate (\ref{gene.est}) is  the general form of spectral estimation that has been extensively investigated by prominent statistical researchers as early as 1950s and 1960s. See, e.g., Grenander (1951) \cite{Gre1951}, Parzen (1957) \cite{Par1957} and Priestley (1962) \cite{Pri1962}.

\subsection{Flat-top kernel choice}
As suggested by Theorem \ref{thm.bias}, in order to achieve favorable asymptotic rates, it is desirable to choose a flat-top kernel $\Lambda$ of higher order, and hence a flat-top function $\lambda$  as smooth as possible according to Proposition \ref{thm.bias}. McMurry and Politis (2004) \cite{McM2004} constructed a member of the flat-top family that is infinitely differentiable, which is defined as

\begin{equation} \label{lam.id}
  \lambda_{ID,b,c}(s)=
   \begin{dcases*} 
      \;\;1 & if $|s|\leq c$,\\
      \;\;\exp(-b\exp(-b/(|s|-c)^2)/(|s|-1)^2) & if $c<|s|<1$,\\
      \;\;0 & if $|s|\geq 1$
   \end{dcases*}
\end{equation} 
where $c\in (0,1]$ determines the region over which $\lambda$ is identically 1, and $b>0$ is a shape parameter, making the transition from $\lambda_{ID,b,c}(c)=1$ to $\lambda_{ID,b,c}(1)=0$ more or less abrupt. 

The function $\text{exp}(-b\exp(-b/(|s|-c)^2)/(|s|-1)^2)$  connects the regions where $\lambda$ is 0 and the region where $\lambda$ is 1 in a manner such that $\lambda(s)$ is infinitely differentiable for all $s$, including where $|s| = c$ and $|s| = 1$. The resulting kernel $\Lambda$ is of infinite order in the sense that $\Lambda(x)$ decays faster than $|x|^{-m}$, for all positive finite $m$, as $|x| \to \infty$. Fig.1 shows the plots of the infinitely differentiable flat-top function $\lambda_{ID,b,c}(s)$ with $b=0.25$ and $c=0.05$, and the resulting kernel $\Lambda(x)$ as well as the corresponding weight function $W_{\lambda}^{(T)}(x)$. Note that the plot of $W_{\lambda}^{(T)}(x)$ can be created from either Equation (\ref{weight}) or (\ref{alter.W}) as their equivalence stated in Section \ref{sec.alt.est}.

\begin{figure}
\center
\includegraphics[width=5in]{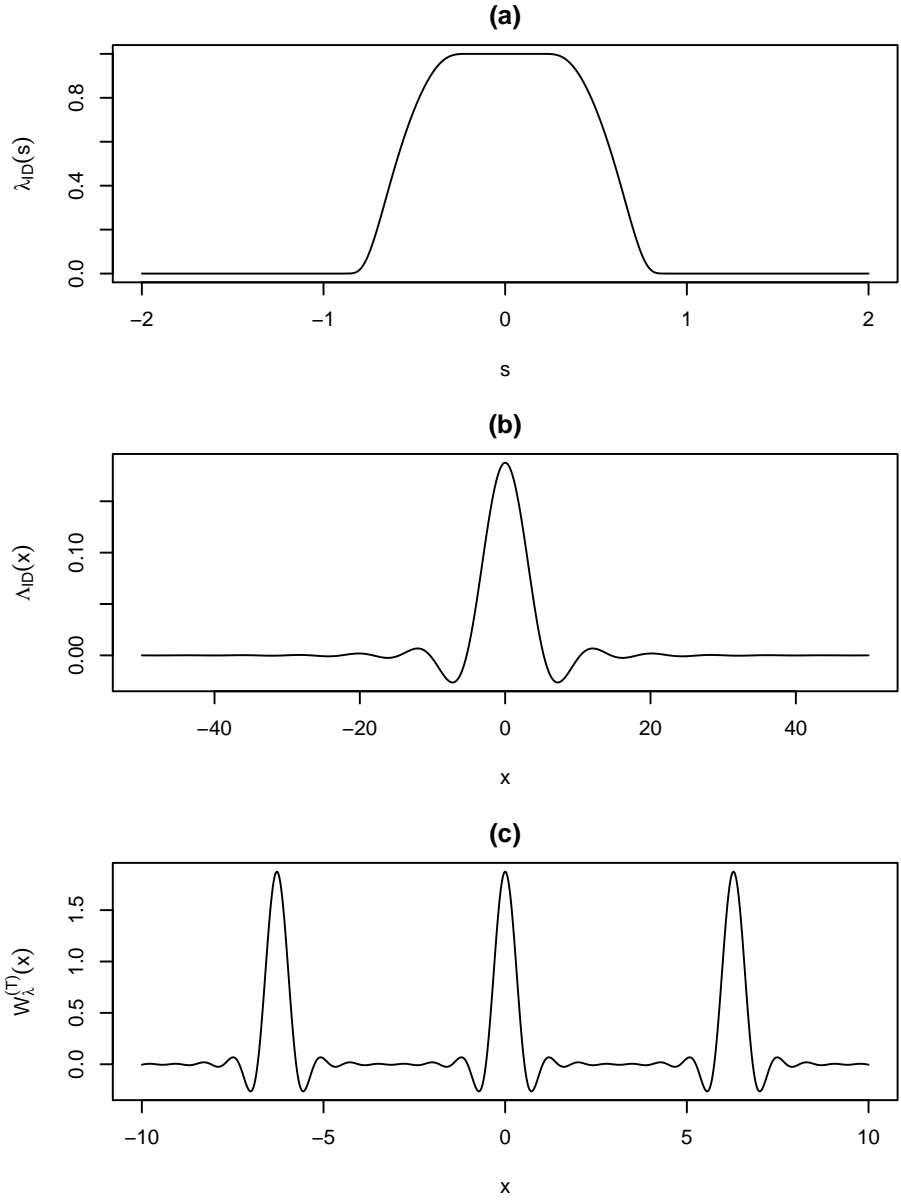} 
\caption{(a) Plot of $\lambda_{ID, 1/4, 0.05}(s)$; (b) Plot of corresponding kernel $\Lambda(x)$ induced by inverse Fourier transform of $\lambda_{ID, 1/4, 0.05}(s)$; (c) Plot of the corresponding weight function $W_{\lambda}^{(T)}$ with $B_T=0.1$. }
\end{figure}

Nevertheless, while the effectiveness of the flat-top kernels is reflected in Theorem \ref{thm.bias} and  \ref{thm.mse},  they in fact provide merely theoretical bounds for the decay rate of bias and IMSE. In the meantime, according to Theorem \ref{thm.bias}, the reduction of bias of the flat-top estimation could potentially be limited by the order of the cumulant condition, which indicates that an infinite-order kernel might not be necessary. That leads us to attempt other choices within the flat-top family. One simple representative flat-top function has the trapezoidal shape defined as 

\begin{equation} \label{lam.tr}
  \lambda_{TR, c}(s)=
   \begin{dcases*} 
      \;\;1 & if $|s|\leq c$,\\
      \;\;\frac{|s|-1}{c-1} & if $c<|s|<1$,\\
      \;\;0 & if $|s|\geq 1$.
   \end{dcases*}
\end{equation}

\noindent
The trapezoidal $\lambda_{TR, c}$ is continuous everywhere and it already exhibits good performance when being implemented for the estimation of spectral density matrix; see Politis (2011) \cite{Pol2011}.  The infinitely differentiable function $\lambda_{ID,b,c}(s)$ looks very much like the trapezoidal $\lambda_{TR, 1/2}$ with ultra-smoothed corners. 

Another choice to be considered is the flat-top function created by adding a piecewise cubic tail, similar to that of Parzen's (1961) \cite{Par1961} kernel, to the $[-c,c]$ flat-top region. It is defined as

\begin{equation} \label{lam.pr}
  \lambda_{PR, c}(s)=
   \begin{dcases*} 
      \;\;1 & if $0\leq s \leq c$,\\
      \;\; 1-6(s-c)^2+6|s-c|^3 & if $c\leq s\leq c+1/2$, \\ 
      \;\; 2(1-|s-c|)^3 & if $c+1/2\leq s \leq c+1$ \\
      \;\; 0 & if $ s\geq c+1$,\\
      \;\; \lambda_{PR, c}(-s) & if $s<0$.
   \end{dcases*}
\end{equation}

\noindent
Plots of flat-top functions $\lambda_{TR,1/2}$ and $\lambda_{PR,3/4}$ are shown in Figure 2. 
 Concerning the choice of parameters of flat-top kernels, i.e., $b$ and $c$, we refer the readers to Politis (2011) \cite{Pol2011} where a detailed discussion is given.

\begin{figure}
\center
\includegraphics[width=6in]{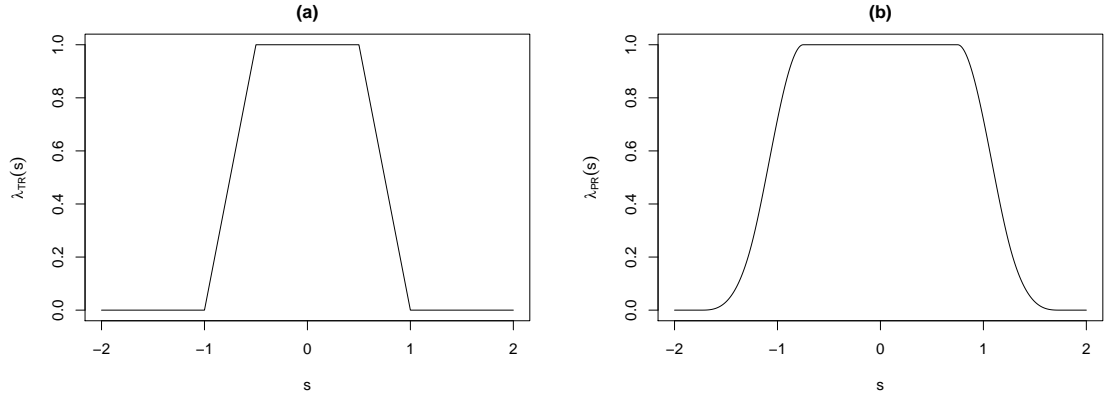} 
\caption{(a) Plot of trapezoidal $\lambda_{TR, 1/2}(s)$; (b) Plot of flat-top Parzen $\lambda_{PR,3/4}(s)$. }
\end{figure}

\section{Positive semi-definite spectral estimation}
By employing the infinite-order flat-top kernels, the flat-top estimator $\hat{\mathscr{F}}_{\omega, \lambda}^{(T)}$ is capable of achieving higher-order accuracy with improved estimation bias. The disadvantage of flat-top kernels, however, is that they are not positive semi-definite. As a result, the operator estimation $\hat{\mathscr{F}}_{\omega,\lambda}^{(T)}$ is not almost surely positive semi-definite for all $\omega$, while it converges to a positive semi-definite operator $\mathscr{F}_{\omega}$. 

The positive semi-definiteness of the estimation is desirable especially in the case of $\omega=0$ when the object is estimation of a long-run covariance operator. In the context of finite-dimensional time series analysis, the spectral density matrix estimators can be easily adjusted to be positive semi-definite via replacing negative eigenvalues by zeros in the diagonalization of the estimated matrices; see e.g. Politis (2011) \cite{Pol2011}. Analogously, we now show how the flat-top operator estimator $\hat{\mathscr{F}}_{\omega,\lambda}^{(T)}$ can be modified to render a positive semi-definite estimator while preserving the asymptotic consistency.

The spectral decomposition of operators in an infinite-dimensional Hilbert space is much more intricate than that in a finite-dimensional context. However, recall that both operators $\mathscr{F}_{\omega}$ and $\hat{\mathscr{F}}_{\omega, \lambda}^{(T)}$ are induced by kernel functions through right integration, and therefore they are symmetric Hilbert-Schmidt operators that admit the following decompositions 
   
\begin{equation}
\mathscr{F}_{\omega}(h)=\sum_{j=1}^{\infty}\nu_j\langle h, e_j\rangle e_j, \qquad h \in L^2([0,1],\mathbb{R})
\end{equation}
\begin{equation}
\hat{\mathscr{F}}_{\omega,\lambda}^{(T)}(h)=\sum_{j=1}^{\infty}\hat\nu_j\langle h, \hat e_j\rangle \hat e_j, \qquad h \in L^2([0,1],\mathbb{R})
\end{equation}

\noindent
where $(\nu_j)$ and $(\hat \nu_j)$ are two sequences of real numbers tending to zero; $(e_j)$ and $(\hat e_j)$ are two orthonormal bases of $L^2([0,1],\mathbb{R})$. We have for $j\geq 1$,

\begin{equation*}
\mathscr{F}_{\omega}(e_j)=\nu_j e_j \quad \text{and} \quad \hat{\mathscr{F}}_{\omega,\lambda}^{(T)}(\hat e_j)=\hat \nu_j \hat e_j;
\end{equation*}

\noindent
thus $(\nu_j, e_j)$ and $(\hat\nu_j, \hat e_j)$, $j\geq 1$ are complete sequences of eigenelements of $\mathscr{F}_{\omega}$ and $\hat{\mathscr{F}}_{\omega,\lambda}^{(T)}$ respectively.

Noting that the eigenvalues $\nu_j, \; j\geq 1$ are all non-negative since the operator $\mathscr{F}_{\omega}$ is positive semi-definite. To fix the possible negativity of $\hat{\mathscr{F}}_{\omega,\lambda}^{(T)}$, let $\tilde \nu_j=\max(\hat\nu_j,0)$ for all $j$, and define the estimator

\begin{equation}
\tilde{\mathscr{F}}_{\omega,\lambda}^{(T)}(h)=\sum_{j=1}^{\infty}\tilde\nu_j\langle h, \hat e_j\rangle \hat e_j, \qquad h \in L^2([0,1],\mathbb{R}).
\end{equation}

\noindent
We keep nonnegative eigenvalues of $\hat{\mathscr{F}}_{\omega,\lambda}^{(T)}$ and replace negative eigenvalues by zero, which makes the resulting operator $\tilde{\mathscr{F}}_{\omega,\lambda}^{(T)}$ an positive semi-definite estimator. The connection of $\hat{\mathscr{F}}_{\omega,\lambda}^{(T)}$ and $\tilde{\mathscr{F}}_{\omega,\lambda}^{(T)}$ is shown in the following inequality:\\

\begin{proposition}\label{prop.semidef}
Let $\tilde{\mathscr{F}}_{\omega,\lambda}^{(T)}$ be the positive semi-definite operator estimator of $\mathscr{F}_{\omega}$ defined in (9), then for a fixed $\omega$
\begin{equation}
\vertiii{\tilde{\mathscr{F}}_{\omega,\lambda}^{(T)} -\mathscr{F}_{\omega}}_2 \leq \vertiii{\hat{\mathscr{F}}_{\omega,\lambda}^{(T)} -\mathscr{F}_{\omega}}_2,
\end{equation}
where $\vertiii{\cdot}_2$ is the Hilbert-Schmidt norm.
\end{proposition}

\begin{proof}
See Appendix.
\end{proof}

A direct consequence of the last result is the following corollary which shows that, in addition to being positive semi-definite, $\tilde{\mathscr{F}}_{\omega,\lambda}^{(T)}$ possesses the same mean square convergence of $\hat{\mathscr{F}}_{\omega,\lambda}^{(T)}$ given in Theorem \ref{thm.mse}.\\

\begin{theorem}
Under the condition of Theorem \ref{thm.mse},  the positive semi-definite spectral density operator estimate $\tilde{\mathscr{F}}_{\omega,\lambda}^{(T)}$ employing a flat-top kernel $\Lambda$ of order $p$  is consistent in integrated mean square with
$$\text{IMSE}(\tilde{\mathscr{F}}_{\omega,\lambda}^{(T)}) = \int_{-\pi}^{\pi}\mathbb{E}\vertiii{\tilde{\mathscr{F}}_{\omega,\lambda}^{(T)}-\mathscr{F}_\omega}_2^2d\omega =O(B_T^{2p})+O(B_T^{-1}T^{-1})$$
where $\vertiii{\cdot}_2$ is the Hilbert-Schmidt norm.
\end{theorem}

In the case that the estimand $\mathscr{F}_\omega$ is not only positives semi-definite but strictly positive definite, it is desirable to have a strictly positive definite estimator of $\mathscr{F}_\omega$. A similar modification of $\hat{\mathscr{F}}_{\omega,\lambda}^{(T)}$ can be applied here to make the estimator strictly positive definite. Let  $\check\nu_j=\max(\check\nu_j,\epsilon_T)$ for all $j$, where $\epsilon_T>0$ is some chosen sequence, and define the estimator

\begin{equation}
\check{\mathscr{F}}_{\omega,\lambda}^{(T)}(h)=\sum_{j=1}^{\infty}\check\nu_j\langle h, \hat e_j\rangle \hat e_j, \qquad h \in L^2([0,1],\mathbb{R}).
\end{equation}

The estimator $\check{\mathscr{F}}_{\omega,\lambda}^{(T)}$ is positive definite and it can be verified that it maintains the high accuracy of the flat-top estimator if $\epsilon_T=O(1/T)$. Thus, $\check{\mathscr{F}}_{\omega,\lambda}^{(T)}$ is a higher-order accurate, strictly positive definite estimator.

\section{Data-dependent bandwidth choice}

As it has been demonstrated in Section \ref{sec.alt.est} that the lag-window estimate (\ref{gene.est}) with bandwidth $m_T$ is nearly equal to the estimate (\ref{est}) with bandwidths $B_T=m_T^{-1}$, we propose here an empirical rule for choosing the bandwidth $B_T$ in practice, which resembles the bandwidth choosing rule for the flat-top lag-window introduced in Politis (2011) \cite{Pol2011}.

Recall that the sample autocovariance kernel 
\begin{equation*}
\hat r_u(\tau,\sigma)= \frac{1}{T}\sum_{0\leq t,t+u\leq T-1}X_{t+u}(\tau)X_{t}(\sigma),
\end{equation*}
the proposed bandwidth choice rule is done by a simple inspection of the functional version of correlogram/cross-correlogram, i.e. a plot of $\hat\rho_m(\tau,\sigma)$ vs. $m$ where
\begin{equation*}
\hat\rho_m(\tau,\sigma)= \frac{\hat r_m(\tau,\sigma)}{\sqrt{\hat r_0(\tau,\tau)\hat r_0(\sigma,\sigma)}}
\end{equation*}
for all $\tau,\sigma\in [0,1]$.

We look for a point, say $\hat q$, after which the correlogram for each pair of $(\tau, \sigma)$ appears negligible, i.e. $\hat\rho_m(\tau,\sigma)\simeq 0$ for $|m|>\hat q$, and $\hat\rho_{\hat q}(\tau,\sigma)\neq 0$. Here $\hat\rho_m(\tau,\sigma)\simeq 0$ is taken to mean that $\hat\rho_m(\tau,\sigma)$ is not taken significantly different from 0. 
In practice, we determine $\hat q$ by considering the correlogram for $(\tau,\sigma)$ over a finite grid of $[0,1]\times[0,1]$.
After identifying $\hat q$, the recommendation is to take 
\begin{equation}
\hat B_T = \frac{1}{\max(\lceil\hat q/c\rceil,1)}
\end{equation}
where $c$ is the parameter determines the `flat-top' region of $\lambda$.

From the flat-top lag-window perspective, the intuition behind the above bandwidth choice rule is an effort to extend the `flat-top' region of $\lambda$ over the whole of the region where $\hat\rho_{\hat q}(\tau,\sigma)$ is thought to be significant so as not to downweigh it and introduce bias. As scrutinized in Politis (2011) \cite{Pol2011}, the `flat-top' region of $\lambda$ can be greater than $[-c, c]$ depending on the choice of function $g$. The decreasing rate of $g(s)$ near $c$ could be slow enough so that $\lambda(s)\simeq 1$ for an interval much greater than $[-c, c]$; see, for example, (\ref{lam.id}) and Figure 1(a) regarding the infinitely differentiable $\lambda_{ID,b,c}(s)$ with $b=0.25$ and $c=0.05$. Instead of the interval $[-c,c]$, we consider an `effective' flat-top region of $\lambda$ defined as the interval $[-c_{ef},c_{ef}]$ where $c_{ef}$ is the largest number such that $\lambda(s)\geq 1-\epsilon$ for all $x$ in $[-c_{ef},c_{ef}]$; here $\epsilon$ is some small number chosen number, e.g. $\epsilon =0.01$.

Let $\Gamma = \{(i/10,j/10); i,j=0,\ldots,9)\}$ be a finite grid of $[0,1]^2$. Now we can formalize the empirical rule of choosing bandwidth $B_T$.\\

\noindent
{\bf{EMPIRICAL RULE OF CHOOSING BANDWIDTH $B_T$.}} 

\vspace{2mm}
\noindent
{\em{For $(\tau,\sigma) \in \Gamma$, let $\hat q_{\tau,\sigma}$  be the smallest nonnegative integer such that $|\hat \rho_{m+\hat q_{\tau,\sigma}}(\tau,\sigma)|< C_0\sqrt{\log_{10}T/T}$, for $m=0,1,\ldots,K_T$,  where $C_0>0$ is a fixed constant, and $K_T$ is a positive, nondecreasing integer-valued function of $T$ such that $K_T=o(\log T)$. Then, let $\hat q= \underset{{(\tau,\sigma)\in\Gamma}}{\max}\hat q_{\tau,\sigma}$, and $B_T=1/\max(\lceil\hat q/c_{ef}\rceil,1)$.
}}

\vspace{4mm}

The constant $C_0$ and the form of $K_T$ are the practitioner's choice. Politis (2003) \cite{Pol2003} makes the concrete recommendations $C_0\simeq 2$ and $K_T= \max(5, \sqrt{\log_{10}T})$ that have the interpretation of yielding (approximately) 95\% simultaneous confidence intervals for $\hat \rho_{m+\hat q_{\tau,\sigma}}(\tau,\sigma)$ with $m=1,\ldots, K_T$ by Bonferroni's inequality.

It is also worth noting that by considering the correlogram over the finite grid $\Gamma$, we actually generate a matrix of thresholds and,  $\hat q$ is picked as the maximum among the entries of the matrix, i.e, $\hat q_{\tau,\sigma}$ for $(\tau,\sigma) \in \Gamma$. This rule of identifying $\hat p$ can be blemished in the situation that certain $\hat q_{\tau,\sigma}$ are radically greater compared to the others. Picking $\hat p$ to be the average of the matrix entries will be a more reasonable choice when such a special case arises. Nevertheless, if the target is to estimate the spectral kernel $f_{\omega}$ for a particular pair $(\tau,\sigma)$, one can always choose the bandwidth $B_T$ by using the specified $\hat q_{\tau,\sigma}$, i.e. $B_T=1/\max(\lceil\hat q_{\tau,\sigma}/c_{ef}\rceil,1)$.

\section{Simulations}

We now present some numerical simulations to complement our asymptotic results. The main goal of the simulations is to compare the performance of the estimators employing flat-top kernels with that of  the non-flat-top estimation, as well as to illustrate the main issues discussed in the paper.  The simulations are performed on a simple functional moving average model
\begin{equation}\label{fma}
X_t = A_0\varepsilon_t+A_1\varepsilon_{t-1}.
\end{equation}
The simulations we carry out are analogous to that conducted in Panaretos and Tavakoli (2013) \cite{Pan2013}.
The innovation functions $\varepsilon_t$'s are independent Wiener processes on $[0,1]$, which are represented using a truncated Karhuen-Lo\`eve expansion,
\begin{equation*}
\varepsilon_t(\tau)= \sum_{k=1}^{1000}\xi_{k,t}\sqrt{\eta_k}e_k(\tau),
\end{equation*}
where $\eta_k=1/[(k-1/2)^2\pi^2]$, $\xi_{k,t}$ are independent standard Gaussian random variables and $e_k(\tau)=\sqrt{2}\sin[(k-1/2)\pi\tau]$ is orthonormal system in $L^2([0,1],\mathbb{R})$; see Adler (1990) \cite{Adl1990}. The operators $A_0$ and $A_1$ are constructed so that their image be contained within a 50-dimensional subspace of $L^2([0,1],\mathbb{R})$, spanned by an orthonormal basis $\psi_1,\dots,\psi_{50}$. Representing $\varepsilon_t$ in the $e_k$ basis, and $A_s,s=1,2$ in the $\psi_m\bigotimes e_k$ basis, we have the matrix representation of the process $X_t$ as $\bm{X}_t=\bm{A}_0\bm{\varepsilon}_t+\bm{A}_1\bm{\varepsilon}_{t-1}$, where $\bm{X}_t$ is a $50\times 1$ matrix, each ${\bm A}_s$ is a $50\times 100$ matrix, and each $\bm{\varepsilon}_t$ is a $100\times 1$ matrix.

A stretch of $X_t, t=0,\dots, T-1$ is generated for $T=2^n$ with $n=6,\dots, 10$. Matrices ${\bm A}_s, s=1,2$ are constructed as random Gaussian matrices with independent entries, such that element in $j$th row are $N(0,j^{-2})$ distributed. 

For the simulation, $B=200$ simulation runs are generated for each $T$ which are used to compute the IMSE by approximating the integral 
\begin{equation*}
2\int_0^\pi\mathbb{E}\vertiii{\hat{\mathscr{F}}_{\omega,\lambda}^{(T)}-\mathscr{F}_\omega}_2^2d\omega
\end{equation*}
by a weighted sum over the finite grid $\Gamma=\{\pi j/10;j=0,\ldots,9\}$. We consider the estimators with proposed flat-top kernels and compare them with the {\em Epanechnikov kernel}, $W(x)=\frac{3}{4}(1-x^2)^{+}$, which is non-flat-top implemented in  the simulations of Panaretos and Tavakoli (2013) \cite{Pan2013}. We apply bandwidth $B_T=T^{-1/5}$ for the estimator of each kernel. In addition, the bandwidths of the estimators employing flat-top kernels are also estimated using the empirical rule proposed in Section 5.

The simulation results are presented in Table 1, entries of which are logarithm of IMSE in base 2. As expected, the estimators employing flat-top kernels show a faster decay rate of IMSE compared to the one with the non-flat-top {\em Epanechnikov kernel}.
The performance of flat-top Parzen's kernel and flat-top infinitely differentiable kernel is close, while each slightly outperforms the trapezoid as the sample size grows. This might be due to the fact that the smoothness of the flat-top functions is indeed a factor on the decay of IMSE, but over-smoothing might not be necessary as the performance could potentially be limited by the order of cumulant conditions as Theorem \ref{thm.mse} suggests. Also note that implementing the empirical rule of bandwidth choice yields a slight improvement as the sample size grows.

\begin{table} 
    \centering
    \label{tab1}
    \begin{subtable}[b]{6in}
        \centering
    \begin{tabular}{p{5.0cm} || c | c | c | c | c | c}
    \hline
     $T$ & 64 & 128 & 256 & 512 & 1024 & 2048 \\ \hhline{=||=|=|=|=|=|=}
     Epanechnikov kernel & -6.188 & -6.852 &  -7.588 & -8.276 & -9.082 & -9.816  \\ 
    $\lambda_{TR,1/2}$ (Trapezoid) & {\bf-6.389} & -7.112 & -7.902 & -8.719 & -9.493 & -10.146   \\
    $\lambda_{PR,3/4}$ (flat-top Parzen) & -6.383 & -7.041 & -8.018 & {\bf-8.846} & -9.710 & -10.453 \\ 
    $\lambda_{ID,1/4,0.05}$ (flat-top Inf. Diff.) & -6.344 & {\bf-7.262} & {\bf-8.074} & -8.832 & {\bf-9.719} & {\bf -10.470} \\ 
    \hline
    \end{tabular}
       \caption{Bandwidth $B_T=T^{-1/5}$}
            \end{subtable}

   \vspace{2mm}
    \begin{subtable}[b]{6in}
        \centering
    \begin{tabular}{p{5.0cm} || c | c | c | c | c | c}
    \hline
     $T$ & 64 & 128 & 256 & 512 & 1024 & 2048 \\ \hhline{=||=|=|=|=|=|=}
     Epanechnikov kernel & -7.148 & -7.668 &  -8.577 & -9.075 & -9.950 & -10.887 \\ 
    $\lambda_{TR,1/2}$ (Trapezoid) & -7.321 & -7.808 & -8.789 & -9.297 & -10.204& -11.193 \\
    $\lambda_{PR,3/4}$ (flat-top Parzen) & {\bf-7.343} & -8.028 & -8.881 & {\bf-9.562} & -10.306 & {\bf-11.426}\\ 
    $\lambda_{ID,1/4,0.05}$ (flat-top Inf. Diff.) & -7.241 & {\bf-8.216} & {\bf-8.911} & -9.546 & {\bf-10.332} & -11.412  \\ 
    \hline
    \end{tabular}
       \caption{Bandwidth $B_T=2\cdot T^{-1/5}$}
            \end{subtable}
            
   \vspace{2mm}
   \begin{subtable}[b]{6in}
   \centering
    \begin{tabular}{p{5cm} || c | c | c | c | c | c }
    \hline
     $T$ & 64 & 128 & 256 & 512 & 1024 & 2048\\ \hhline{=||=|=|=|=|=|=}
    $\lambda_{TR,1/2}$ (Trapezoid) & -6.519 & -7.331 & -8.260 & -9.145 & -10.089 & -11.008  \\
    $\lambda_{PR,3/4}$ (flat-top Parzen) & {\bf -6.627} & {\bf-7.592} & {\bf-8.455} & {\bf-9.349} & {\bf-10.313} & {\bf-11.371} \\ 
    $\lambda_{ID,1/4,0.05}$ (flat-top Inf. Diff.) & -6.118 & -6.925 & -8.053 & -9.056 & -10.214 & -11.121\\ 
    \hline
    \end{tabular}
       \caption{Empirical rule of choosing $B_T$}
     \end{subtable}

    \caption{ Entries represent the logarithm of IMSEs in base 2 of different estimators using (a) bandwidth $B_T=T^{-1/5}$, (b) bandwidth $B_T=2\cdot T^{-1/5}$ and (c) empirical rule of choosing $B_T$. Sample size ranges from $2^6$ to $2^{11}$. Minimum IMSE for each $T$ is indicated by boldface.}
\end{table}

\section{Appendix: Proofs}
\subsection{Proof of Proposition \ref{prop.order}}
To prove Proposition \ref{prop.order}, we need the following lemma, which measures the magnitude of Fourier coefficient.
\begin{lemma} \label{lemma.fourier}
Let $f$ be an integrable function on the interval $[0,2\pi]$, and $\hat f(n)$ be its Fourier coefficients defined by 
\begin{equation}\label{fourier.lemma}
\hat f(n) = \frac{1}{2\pi}\int_0^{2\pi}f(t)e^{int}dt.
\end{equation}
If $f$ is $p$ times differentiable and $f^{(p)}$ is H\"older continuous of order $o<\alpha<1$, then 
$$|\hat f(n)|=O(|n|^{-p-\alpha}).$$
\end{lemma}

\begin{proof}
By repeated integration by parts on Equation (\ref{fourier.lemma}), we have
\begin{equation}\label{fourier.int}
\hat f(n)=(in)^{-p}\widehat{f^{(p)}}(n).
 \end{equation}
On the other hand, $\widehat{f^{(p)}}(n)=\frac{1}{2\pi}\int_0^{2\pi}f^{(p)}(n)e^{-int}dt=\frac{-1}{2\pi}\int_0^{2\pi}f^{(p)}(n)e^{-in(t+\pi/n)}dt$; by a change of variable, $\widehat{f^{(p)}}(n)$ can be written as
$$ \widehat{f^{(p)}}(n)= \frac{1}{4\pi}\int_0^{2\pi}\left(f^{(p)}(t+\frac{\pi}{n})-f^{(p)}(t)\right)e^{-int}dt.$$
By H\"older continuity of $f^{(p)}$, we have
\begin{equation}\label{fourier.res}
|\widehat{f^{(p)}}(n)|\leq\frac{C}{|n|^{\alpha}},
\end{equation}
for some constant $C$. Combining (\ref{fourier.int}) and (\ref{fourier.res}), we obtain
\begin{equation}
|\hat f(n)|=|n|^{-p}|\widehat{f^{(p)}}(n)|=O(|n|^{-p-\alpha}).
\end{equation}
\end{proof}

\noindent
{\bf{Proof of Proposition 2.1.}} For a flat-top function $\lambda(s)$ and its inverse Fourier transform $\Lambda(x)$, we have
\begin{equation}\label{fourier.trans}
\Lambda(x) =\frac{1}{2\pi}\int_\mathbb{R}\lambda(s)e^{isx}ds,
\end{equation}
\begin{equation}\label{inv.fourier.trans}
\lambda(s) =\int_\mathbb{R}\Lambda(x)e^{-isx}dx.
\end{equation}
Since the assumption that $\lambda$ is $p$ times differentiable and $\lambda^{(p)}$ is H\"older continuous of order $0<\alpha<1$, by Lemma \ref{lemma.fourier}, we have $|x|^{p+\alpha}|\Lambda(x)|\leq C$ for some constant $C$, which implies $\Lambda(x)$ has finite moments up to order $p-1$, i.e. $\int_\mathbb{R}|x|^k|\Lambda(x)|dx<\infty$ for $0\leq k\leq p-1$.\\

\noindent
By repeated differentiations on both sides of (\ref{inv.fourier.trans}),  we obtain for $k = 1,\dots, p-2,$
$$\frac{d^k\lambda(s)}{d s^k} = \int_\mathbb{R}(-ix)^k\Lambda(x)e^{-isx}dx$$
by dominated convergence theorem. Now that $\lambda(x)$ is flat-top, $\lambda^{(k)}(0)$ is zero for all $k$, which in turn leads to 
\begin{equation}\label{lemma.order}
\int_\mathbb{R}x^k\Lambda(x)dx=0 \text{ for }k = 1,\dots, p-2
\end{equation}
if we set $s=0$ on both sides of (\ref{lemma.order}). Therefore, $\Lambda$ is a kernel of order $p-1$.\\

\subsection{Proof of Theorem \ref{thm.bias}}
To prove Theorem \ref{thm.bias}, the following lemmas are necessary.

\begin{lemma}\label{property.WT}
We have the following properties for the flap-top kernel $\Lambda(x)$ and the function $W_\lambda^{(T)}(x)$:\\

\noindent
(i) $\int_{\mathbb R}\Lambda(x)dx=1;$\\

\noindent
(ii) $\int_{-\pi}^{\pi}W_\lambda^{(T)}(x)dx=1;$\\

\noindent
Let $||f||_\infty=\text{sup}_{x\in[a,b]}|f(x)|$, and denoted by $V_a^b(h)$ the total variation of a function $h:[a,b]\to \mathbb{C}$.\\

\noindent
(iii) If $B_T<1$, $||W_\lambda^{(T)}||_{\infty}=\frac{1}{B_T}||\Lambda||_{\infty}+O(B_T)$;\\

\noindent
(iv) If $B_T<1$, $V_{-\pi}^{\pi}(W_\lambda^{(T)})\leq \frac{1}{B_T}V_{-\pi}^{\pi}(\Lambda).$
\end{lemma}

\begin{proof}
The statement $(i)$ follows directly from setting $s=0$ on both sides of Equation (\ref{inv.fourier.trans}). 
The statement $(ii)$ is obtained by following the same arguments in the proof of Lemma F.11 in Panaretos and Tavakoli (2013) \cite{Pan2013}. 
For the third statement, recall that
\begin{equation*}
W_\lambda^{(T)}(x)=\sum_{j\in \mathbb{Z}}\frac{1}{B_T}\Lambda\left(\frac{x+2\pi j}{B_T}\right).
\end{equation*}
If $B_T<1$, then for $x\in [-\pi,\pi],$
\begin{align*}
|W_\lambda^{(T)}(x)| &\leq  \frac{1}{B_T}  \sum_{j\in \mathbb{Z}}\left|\Lambda\left(\frac{x+2\pi j}{B_T}\right)\right|\\
&=\frac{1}{B_T}\left|\Lambda\left(\frac{x}{B_T}\right)\right| + \frac{1}{B_T}  \sum_{j\in \mathbb{Z}^{+}}\left|\Lambda\left(\frac{x+2\pi j}{B_T}\right)\right| + \frac{1}{B_T}  \sum_{j\in \mathbb{Z}^{-}}\left|\Lambda\left(\frac{x+2\pi j}{B_T}\right)\right|.
\end{align*}
Consider the second term, by the fact that $|\Lambda(x)| \leq \frac{C}{|x|^{1+\alpha}}$ for large $x$, we obtain
\begin{align*}
\frac{1}{B_T}  \sum_{j\in \mathbb{Z}^{+}}\left|\Lambda\left(\frac{x+2\pi j}{B_T}\right)\right| & \leq \frac{1}{B_T}\sum_{j=1}^{\infty}\frac{C}{\left(\frac{x+2\pi j}{B_T}\right)^{1+\alpha}}\\
& = B_T\sum_{j=1}^{\infty}\frac{C}{(x+2\pi j)^{1+\alpha}}\\
& = O(B_T)
\end{align*}
Similarly, the third term above is $O(B_T)$. Therefore, we have $|W_\lambda^{(T)}(x)| \leq \frac{1}{B_T} |\Lambda(x/B_T)| + O(B_T)$. The statement $(iii)$ then follows from the periodicity of $W_\lambda^{(T)}$.\\

\noindent
For the last statement, since
\begin{equation*}
W_\lambda^{(T)}(x)=\frac{1}{B_T}\Lambda\left(\frac{x}{B_T}\right)+ \sum_{j\in \mathbb{Z }, j\neq 0}\frac{1}{B_T}\Lambda\left(\frac{x+2\pi j}{B_T}\right),
\end{equation*}
then by the triangle inequality of total variation, we have
$$V_{-\pi}^{\pi}(W_\lambda^{(T)}) \leq V_{-\pi B_T}^{\pi B_T}(\Lambda / B_T)\leq  V_{-\pi }^{\pi}(\Lambda / B_T)= \frac{1}{B_T}V_{-\pi}^{\pi}(\Lambda)$$
where the second inequality holds because $B_T<1$. Here we use several properties of total variation, see Lemma F.6 in Panaretos and Tavakoli (2013) \cite{PanTav2013}.
\end{proof}

\begin{lemma}\label{lemma.WT.approx}
If $B_T\to 0$,
\begin{equation*}
\frac{2\pi}{T}\sum_{s=1}^{T-1}W_\lambda^{(T)}(\omega-2\pi s/T)=1+O(B_T^{-1}T^{-1}).
\end{equation*}
\end{lemma}

\begin{proof}
Let
\begin{align*}
\Delta_n & =\int_{-\pi}^\pi W_\lambda^{(T)}(\omega-\alpha)d\alpha-\frac{2\pi}{T}\sum_{s=1}^{T-1}W_\lambda^{(T)}(\omega-2\pi s/T)\\
& = \int_{-\pi}^\pi {W'}_\lambda^{(T)}(\alpha)d\alpha-\frac{2\pi}{T}\sum_{s=1}^{T-1}{W'}_\lambda^{(T)}(2\pi s/T)
\end{align*}
where ${W'}_\lambda^{(T)}(\alpha)={W}_\lambda^{(T)}(\omega-\alpha)$. We have
\begin{align*}
|\Delta_n| &\leq \frac{2\pi}{T}\left\{V_{-\pi}^\pi({W'}_\lambda^{(T)})+||{W'}_\lambda^{(T)}||_{\infty}\right\}\\
& \leq \frac{2\pi}{T}\left\{\frac{1}{B_T}V_{-\pi}^\pi({\Lambda})+\frac{1}{B_T}||\Lambda||_{\infty}+O(B_T)\right\}\\
&= \frac{2\pi}{B_T T}\left\{V_{-\pi}^\pi({\Lambda})+||\Lambda||_{\infty} + O(B_T^2)\right\}.
\end{align*}
where the first inequality above follows from Lemma F.10 in Panaretos and Tavakoli (2013) \cite{PanTav2013} and the second inequality follows from $(iii)$ and $(iv)$ of Lemma \ref{property.WT}. Hence, $|\Delta_n|$ is of order $O(B_T^{-1}T^{-1})$ as $B_T\to 0$. Then by Lemma \ref{property.WT}$(ii)$, we obtain
\begin{equation*}
\frac{2\pi}{T}\sum_{s=1}^{T-1}W_\lambda^{(T)}(\omega-2\pi s/T)=1+O(B_T^{-1}T^{-1}).
\end{equation*}
\end{proof}

\noindent
{\bf{Proof of Theorem \ref{thm.bias}.}} By Proposition 2.6 in Panaretos and Tavakoli (2013) \cite{Pan2013}, under {\em C(0,2)} the periodogram kernel has expectation

\begin{equation*}
\mathbb{E}[p_{2\pi s/T}^{(T)}(\tau,\sigma)]=f_{2\pi s/T}(\tau,\sigma)+O(T^{-1}),
\end{equation*}
then we can write
\begin{align*}
\mathbb{E}[\hat f_{\omega,\lambda}^{(T)}(\tau,\sigma)] &= \frac{2\pi}{T} \sum_{s=1}^{T-1}W_\lambda^{(T)}\left(\omega-\frac{2\pi s}{T}\right)   \left\{f_{2\pi s/T}(\tau,\sigma)+O(T^{-1})\right\}=A+B,
\end{align*}
where
\begin{align*}
A &=  \frac{2\pi}{T} \sum_{s=1}^{T-1}W_\lambda^{(T)}\left(\omega-\frac{2\pi s}{T}\right)f_{2\pi s/T}(\tau,\sigma),\\
B &= O(T^{-1})\left\{ \frac{2\pi}{T} \sum_{s=1}^{T-1}W_\lambda^{(T)}\left(\omega-\frac{2\pi s}{T}\right) \right\}.
\end{align*} 
Noting that $||W_\lambda^{(T)}||_\infty = O(B_T^{-1})$ by Lemma \ref{property.WT}, and $||f_.||_\infty=O(1)$, it follows from Lemma F.6(i) and F.10 of Panaretos and Tavakoli (2013) \cite{PanTav2013} that 
$$A=\int_0^{2\pi}W_\lambda^{(T)}(\omega-\alpha)f_\alpha(\tau,\sigma)d\alpha + \varepsilon_T$$
where $\varepsilon_T\sim O(B_T^{-1}T^{-1})$, uniformly in $\omega$. Using Lemma \ref{lemma.WT.approx}, $B=O(T^{-1})$ if $B_TT\to \infty$. Combining these facts, and with a change of variable $\alpha=\omega-xB_T$ on the integral, we obtain 
\begin{align}
\mathbb{E}[\hat f_{\omega,\lambda}^{(T)}(\tau,\sigma)] &= \int_0^{2\pi}W_\lambda^{(T)}(\omega-\alpha)f_\alpha(\tau,\sigma)d\alpha+O(B_T^{-1}T^{-1}) + O(T^{-1})\nonumber\\
&=\int_{\mathbb{R}}\Lambda(x)f_{\omega-xB_T}dx +O(B_T^{-1}T^{-1}). \label{proof.bias1}
\end{align}

\noindent
Following the similar lines in the proof of Lemma F.4 in Panaretos and Tavakoli (2013) \cite{Pan2013}, we can use the Taylor expansion of $f_{\omega-xB_T}$ to obtain 
\begin{align*}
\int_{\mathbb{R}}\Lambda(x)f_{\omega-xB_T}dx = & f_\omega  + \sum_{k=1}^{p-1}\frac{(-1)^kB_T^k}{k!}\cdot\frac{\partial^kf_\omega}{\partial \omega^k}\cdot \int_{\mathbb{R}}x^k\Lambda(x)dx\\
&+\frac{B_T^p}{p!}\cdot\text{sup}\left\Vert\frac{\partial^kf_\omega}{\partial \omega^k}\right\Vert\cdot\int_{\mathbb{R}}|x|^p\Lambda(x)dx.
\end{align*}
With $\Lambda$ being a kernel of order $p$, i.e. $\int_{\mathbb{R}}x^k\Lambda(x)dx=0 $ for all $k=1,\dots,p-1$, we obtain
\begin{equation}\label{proof.bias2}
\int_{\mathbb{R}}\Lambda(x)f_{\omega-xB_T}dx=f_\omega+O(B_T^p).
\end{equation}

\noindent
Combining (\ref{proof.bias1}) and (\ref{proof.bias2}) completes the proof.

\subsection{Proof of Proposition \ref{prop.semidef}}
Recall the spectral decompositions of the operators $\mathscr{F}_{\omega}$, $\hat{\mathscr{F}}_{ \omega,\lambda}^{(T)}$ and $\tilde{\mathscr{F}}_{\omega,\lambda}^{(T)}$. For $h \in L^2([0,1],\mathbb{R})$,
\begin{equation*}
\mathscr{F}_{\omega}(h)=\sum_{j=1}^{\infty}\nu_j\langle h, e_j\rangle e_j, 
\end{equation*}
\begin{equation*}
\hat{\mathscr{F}}_{\omega,\lambda}^{(T)}(h)=\sum_{j=1}^{\infty}\hat\nu_j\langle h, \hat e_j\rangle \hat e_j,
\end{equation*}
\begin{equation*}
\tilde{\mathscr{F}}_{\omega,\lambda}^{(T)}(h)=\sum_{j=1}^{\infty}\tilde\nu_j\langle h, \hat e_j\rangle \hat e_j.
\end{equation*}

\noindent
Noting that $\mathscr{F}_{\omega}$ and $\hat{\mathscr{F}}_{\omega,\lambda}^{(T)}$ are both induced by kernel functions through right integration, the two operators commute, i.e., $\mathscr{F}_{\omega} \hat{\mathscr{F}}_{\omega,\lambda}^{(T)} =  \hat{\mathscr{F}}_{\omega,\lambda}^{(T)} \mathscr{F}_{\omega}$. Due to the fact that commuting operators can be simultaneously diagonalised, $\mathscr{F}_{\omega}$ and $\hat{\mathscr{F}}_{\omega,\lambda}^{(T)}$ share the common eigenfunctions, i.e., $e_j=\hat e_j$ for all $j$. Therefore, eigenvalues of $\hat{\mathscr{F}}_{\omega,\lambda}^{(T)}-\mathscr{F}_{\omega}$ are $\hat \nu_j -\nu_j$, $j\geq1$; and eigenvalues of $\tilde{\mathscr{F}}_{\omega,\lambda}^{(T)}-\mathscr{F}_{\omega}$ are $\tilde \nu_j -\nu_j$, $j\geq1$.\\

\noindent
Viewed as an estimator of the nonnegative $\nu_j$, $\tilde\nu_j$ is a better estimator than $\hat\nu_j$ in the sense that $|\tilde\nu_j-\nu_j|\leq |\hat\nu_j-\nu_j|$ always holds true. Hence, it follows that 

\begin{equation*}
\vertiii{\tilde{\mathscr{F}}_{\omega,\lambda}^{(T)} -\mathscr{F}_{\omega}}_2^2 =
\sum_{j=1}^{\infty}(\tilde \nu_j -\nu_j)^2 \leq
\sum_{j=1}^{\infty}(\hat \nu_j -\nu_j)^2   =
\vertiii{\hat{\mathscr{F}}_{\omega,\lambda}^{(T)} -\mathscr{F}_{\omega}}_2^2,
\end{equation*}
which completes the proof.

\

\end{document}